\documentclass{tran-l}

\newtheorem{theorem}{Theorem}[section]
\newtheorem{proposition}[theorem]{Proposition}
\newtheorem{cor}[theorem]{Corollary}
\newtheorem{lemma}[theorem]{Lemma}
\newtheorem*{claim}{Claim}
\theoremstyle{definition}
\newtheorem{definition}[theorem]{Definition}

\theoremstyle{remark}
\newtheorem{remark}[theorem]{Remark}

\numberwithin{equation}{section}

\newcommand{\Iff}{\mbox{$\Longleftrightarrow$}}

\def\proof{{\bf Proof.}\ }
\def\ie{\emph{i.e.}}
\def\Ie{\emph{I.e.}}
\def\pes{\emph{e.g.}}
\def\Pes{\emph{E.g.}}
\def\wlg{without loss of generality}

\def\F{{\mathcal F}}

\def\P{{\mathcal P}}

\def\RR{{\mathcal{R}}}
\def\RP{{\mathcal{R}^{+}}}
\def\U{{\mathcal U}}

\def\Ab{{\mathbf{A}}}

\def\N{{\mathbb N}}

\def\R{{\,\mathbb R}}
\def\Z{{\mathbb Z}}

\def\WW{{\mathbb{W}}}

\def\TT{{\mathbb{T}}}

\def\WW{{\mathbb{W}}}
\def\II{{\mathbb{I}}}

\def\LL{{\mathbb{L}}}
\def\XX{{\mathbb{X}}}

\def\fg{{\varphi}}

\def\og{{\omega}}

\def\sg{{\sigma}}
\def\Sg{{\Sigma}}

\def\min{\mbox{\rm min}\;}
\def\max{\mbox{\rm max}\;}

\def\+#1{\vec{#1}}

\def\lA{{\mathbf{A}}}

\def\ib{{\mathbf{i}}}

\def\xb{{\mathbf{x}}}

\def\aa{{\mathbf{a}}}
\def\bb{{\mathbf{b}}}

\def\nk{{\mathfrak{n}}}

\def\Ik{{\mathfrak{I}}}

\def\Jk{{\mathfrak{J}}}

\def\ng{{\mathfrak{n}}}

\def\Ng{{\mathfrak{N}}}
\def\Nk{{\mathfrak{N}}}

\def\Rg{{\mathfrak{R}}}

\def\*{\times}

\def\0{\emptyset}
\def\7{\setminus}
\def\_{\overline}
\def\eq{\approx}
\def\<{\prec}

\def\ult#1#2{^{#1}_{\; #2}}

\def\incl{\subseteq}

\def\pincl{\subset}

\def\la{\langle}
\def\ra{\rangle}

\def\equ{\approx_{\U}}

\def\zfc{\textsf{ZFC}}

\def\qed{\hfill $\Box$}

\def\SP{$\mathsf{SP} $}
\def\DP{$\mathsf{DP} $}
\def\PP{$\mathsf{PP} $}
\def\AP{$\mathsf{AP} $}
\def\ZP{$\mathsf{ZP} $}
\def\TP{$\mathsf{TP} $}
\def\UP{$\mathsf{UP} $}
\def\NP{$\mathsf{NP} $}
\def\FAP{$\mathsf{FAP} $}

\def\diser{$(*)$}

\begin{document}

\title{Natural Numerosities of sets of tuples}


\author{Marco Forti}
\address{Dipart. di Matematica Applicata ``U. Dini'' - Via Buonarroti 
1C - 56100 PISA (Italy)}
\curraddr{}
\email{forti@dma.unipi.it}
\thanks{Research
partially supported  by  MIUR Grant PRIN 2009, Italy.}

\author{Giuseppe Morana Roccasalvo}
\address{}
\curraddr{}
\email{moranaroccasalvo@mail.dm.unipi.it}
\thanks{}

\subjclass[2010]{Primary 03E65,  03F25; secondary 03A05, 03C20.}

\date{}

\dedicatory{}


\begin{abstract}
    
We consider  a notion of ``numerosity''
for sets of tuples of natural numbers,
that satisfies the \emph{five common notions} of Euclid's Elements, 
so it can agree
with cardinality only for finite sets.
By suitably axiomatizing such a notion, we show that, contrasting to 
cardinal arithmetic,
the natural ``Cantorian'' definitions of  order relation and 
arithmetical operations  provide a very good algebraic
structure. In fact, numerosities can be taken as the non-negative part
of a \emph{discretely ordered ring}, namely the quotient of a formal power 
series ring modulo a suitable (``gauge'') ideal. In particular, special 
 numerosities, called ``natural'', can be identified with the semiring of
 hypernatural numbers  of appropriate ultrapowers of $\N$. 
    
\end{abstract}

\maketitle
%
 \section*{Introduction}

 In this paper we consider a notion of ``equinumerosity'' on 
  sets of tuples
 of natural numbers, \ie\ 
  an equivalence relation, finer than equicardinality,
 that satisfies the celebrated five Euclidean common notions about 
 \emph{magnitudines} (\cite{eu}), including the
principle that ``the whole is greater than the part''. 
  This notion preserves the basic properties of
 equipotency for \emph{finite} sets. In particular,  the natural
 Cantorian definitions endow 
 the corresponding numerosities with a structure of  
 \emph{discretely ordered semiring}, where $0$ is the size of the emptyset,
 $1$ is the size of every singleton,
and greater numerosity corresponds to supersets.

 This idea of ``Euclidean'' numerosity 
 has been recently investigated by V.~Benci, M.~Di Nasso and M.~Forti
 in a sequel of papers, starting with \cite{bd} (see also 
 \cite{8fold}). 
In particular, in \cite{bdf}
 the size of arbitrary sets (of ordinals)
was approximated by means of directed unions of finite sets, while in 
\cite{df} finite dimensional point sets over the real line, and in 
\cite{logan} entire ``mathematical universes'' are 
considered.

 Here we focus on \emph{sets of tuples
 of natural numbers}, and we show that the existence
 of equinumerosity relations for such sets
 is equivalent to the existence of a  class
 of \emph{prime ideals}, named ``gauge'', of a special ring of \emph{formal power 
 series in countably many indeterminates}. 
Similarly to all papers quoted above, special 
ultrafilters are used in order to model numerosities.
 In fact, the gauge ideals corresponding to the equinumerosities called 
``natural'', are in biunique correspondence with \emph{special ultrafilters} 
over the set $\N^{<\og}$ of all finite subsets of $\N$. These
ultrafilters, also called ``gauge'', may be
 of independent interest, being not clear their connection with
 other classes
 of special ultrafilters  considered in the
 literature. In fact, we only prove here that all selective 
 ultrafilters are gauge. 

 \medskip
 The paper is organized as follows. In Section \ref{uno} we derive 
 from the Euclidean common notions a few set theoretic
  principles that are the basis of our formal definition of \emph{equinumerosity
 relation}. Then we prove that the corresponding \emph{numerosities} 
 form an ordered semiring that can be 
 obtained as the 
 non-negative part of the quotient of a ring $\RR$ of formal power 
 series in countably many indeterminates, modulo suitable prime ideals.
 In Section \ref{ult} we show how  numerosities can be embedded in 
 rings of \emph{hyperreal numbers} obtained by means of special ultrafilters.
 In particular, all ``natural'' numerosities are essentially 
 \emph{nonstandard natural numbers}. Actually, they arise through 
 isomorphisms with ultrapowers $\N\ult{\N^{<\og}}{\U}$ modulo ``gauge 
 ultrafilters''.
 A few final remarks and open questions are contained in Section 
 \ref{froq}.
 
\medskip 
 In general, we refer to
  \cite{ck} for definitions and facts
concerning ultrapowers, ultrafilters, and nonstandard models that are 
used in this paper.

\medskip
 The authors are grateful to Andreas Blass and Mauro Di Nasso for useful 
 discussions.

\section{Equinumerosity of point sets}\label{uno}

In this section we study a notion of ``numerosity''
for \emph{point sets of natural numbers},
\emph{i.e.} for sets of tuples of natural numbers.
This numerosity will be defined by starting from
an equivalence relation of ``equinumerosity'', denoted by $\approx$, that
satisfies the basic properties of \emph{equipotency}
between finite sets.

In particular we assume first that equinumerosity satisfies the 
following\footnote{~
We call ``Aristotelian'' this principle, because the favourite 
example of ``axiom'' quoted by Aristotle is ``if equals be subtracted from equals, the
remainders are equal''.}
\begin{description}
    
\item[\AP\ (Aristotelian Principle)]
$A\approx B$ if and only if $A\setminus B\approx B\setminus A$.
\end{description}
The axiom \AP\ is  a compact
equivalent set theoretic formulation of the second and third common notions 
of Euclid's Elements:
\emph{``If equals be added to equals,
the wholes are equal''}, and
\emph{``if equals be subtracted from equals, the
remainders are equal''}. (A precise statement of this equivalence is 
given in Proposition \ref{sumdif} below.)
On the other hand, for infinite sets, the
Aristotelian Principle is clearly incompatible with the Cantorian 
notion of equicardinality.

Notice that the first common notion \emph{``things which are equal to the same 
thing are also equal to one another''} is already secured by the assumption 
that equinumerosity is an equivalence relation.

 Together with the notion of ``having the same numerosity'', it is 
 natural to introduce a ``comparison of numerosities'', so as to 
 satisfy the fifth Euclidean common notion 
 \emph{``the whole is greater than the 
part''}. Of course, this comparison must be coherent with 
equinumerosity, so we are led to the following 

{\bf Definition}
    We say that $A$ is \emph{greater than} $B$, denoted by $A\succ 
    B$, or equivalently that $B$ is \emph{smaller than} $A$, denoted by 
    $B\prec A$, if 
    there exist $A', B'$ such that $A'\supset B'$, $A\eq A'$, and $B\eq 
    B'$.

 The  natural
idea that numerosities of sets are always comparable, combined with 
the fifth Euclidean notion, gives the following trichotomy property:

\begin{description}
\item[\ZP\ (Zermelian Principle)]
Exactly one of the following three conditions holds:
\begin{enumerate}
\item[(a)]
$A\approx B$\,;
\item[(b)]
$A\succ B$;
\item[(c)]
$A\prec B$.
\end{enumerate}
\end{description}

We shall see below that the Zermelian Principle implies that given two sets 
one is equinumerous to some 
superset of the other, and obviously 
that  no proper subset is equinumerous to the  
set itself.(See Proposition \ref{preorder}).

We stress that both properties \AP\ and \ZP\ hold
 for equipotency between finite sets, 
while both  fail badly for equipotency between infinite sets.
So we cannot assume that equipotent sets are always equinumerous, but 
we have to give a suitable interpretation of the fourth Euclidean 
common notion 
\emph{``things applying [exactly] onto one 
 another are equal to one another''}.\footnote{~
 See the accurate discussion of this Euclidean common notion by T. 
 Heath in \cite{eu}.}
 We choose two kinds of numerosity preserving bijections, namely 
 ``permutations of components'' and ``rising dimension''
\begin{description}
\item[\TP\ (Transformation Principle)]
\emph{If $T$ is $1$-$1$, and $T(a)=(a_{\tau 1},\ldots,a_{\tau k})$ is a 
permutation of $a=(a_1,\ldots,a_k)$ for all $a\in A$, 
then $A\approx T[A]$.}
\medskip
\item[\UP\ (Unit Principle)]
$A\* \{n\}\approx A$ for all $n\in\N$.
\end{description}

Remark that the Unit Principle cannot be consistently assumed for all 
point sets $A$. In fact, if $A=\{ n, (n,n), (n,n,n), \ldots \}$, then 
$A\*\{n\}$ is a \emph{proper subset} of $A$, and so it cannot be 
equinumerous to $A$. So we have to restrict the principle \UP. In 
view of the further developements, in particular in order to obtain a 
\emph{semiring} of numerosities, a
convenient choice is the family of all ``finitary'' point sets, 
\ie\ sets that contain only finitely many tuples for each finite 
set of components. 
Denote by $$\WW=\{A\incl\P(\bigcup_{k\in\N}\N^{k})\mid \forall n\exists 
h \forall k>h (A\cap\{0,\ldots, n\}^{k}=\0)\}$$ the family of all \emph{finitary point sets}.
Remark that $\WW$ is a proper superset of the family of all ``finite 
dimensional'' point sets
$\WW_{0}=\bigcup_{d\in\N}\P(\bigcup_{k=1}^{d}\N^{k})$ that has been 
considered in \cite{QSU,tesi}.

Finally, in order to define a product of numerosities, we could introduce 
the following principle 
\begin{description}
\item[\PP\ (Product Principle)]
\emph{If  $A\approx A'$ and $B\approx B'$ then $A\times B\approx A'\times B'$.}
\end{description}

In order to make $\WW$ closed under Cartesian products, we follow the usual practice,
and we identify Cartesian products with the corresponding
``concatenations''. That is, for every $A\subseteq\N^k$
and for every $B\subseteq\N^h$, we identify
$$A\times B=\{((a_1,\ldots,a_k),(b_1,\ldots,b_h))
\mid (a_1,\ldots,a_k)\in A, (b_1,\ldots,b_h)\in B\}$$
with:
$$A\times B=
\{(a_1,\ldots,a_k,b_1,\ldots,b_h)\mid
(a_1,\ldots,a_k)\in A\ \text{and}\ (b_1,\ldots,b_h)\in B\}.$$

With this convention, we have that 
$A\*\{x_{1},x_{2},\ldots,x_{k}\}=A\*\{x_{1}\}\*\{x_{2}\}\*\ldots\*\{x_{k}\}$, 
and so, using also the Transformation 
Principle, we obtain the general property
$$\{P\}\* A\eq A\*\{P\}\eq A\ \mbox{for any point}\ P\in\N^{k}.$$
In particular, any two singletons are equinumerous.

However, assuming this convention, the Product Principle cannot be 
consistently postulated in the above formulation for all 
sets in $\WW$, because different pairs of tuples may share the same 
concatenation.
\Pes\ both $((1,2),(3,4,5))$ and $((1,2,3),(4,5))$ produce 
$(1,2,3,4,5)$. So one should consider concatenated products as 
``multisets'', where each tuple comes with its (finite) 
``multiplicity''. We prefer to consider only pure sets, so we 
restrict the Product Principle to suitably defined ``multipliable pairs''.

Let us call the sets $A,B\in\WW$ \emph{multipliable} if different 
pairs $(a,b)\in A\* B$ have different concatenations. (For instance, 
every set of $\WW_{0}$ is multipliable with every subset of  
 $\N^{k}$.) We 
shall restrict \PP\ to products of \emph{multipliable} sets.
We can now give our precise definition of equinumerosity relation.

\begin{definition}\label{eqn}
Let $$\WW=\{A\incl\P(\bigcup_{k\in\N}\N^{k})\mid \forall n\exists 
h \forall k>h\, (A\cap\{0,\ldots, n\}^{k}=\0)\}$$ be the 
set of all \emph{finitary point sets}.

\noindent
$\bullet$ An equivalence relation 
$\approx$ on $\WW$ is an \emph{equinumerosity} 
if  the following properties are fullfilled
for all  $A, B\in\WW$:

\begin{description}
\item[\AP]
$A\approx B$ if and only if $A\setminus B\approx B\setminus A$.
\item[\ZP]
Exactly one of the following three conditions holds:
$$ \mbox{either}\ A\approx B, \ \ \mbox{or}\
 A\succ B,\ \ \mbox{or}\
 A\prec B,$$
where $A$ is \emph{greater than} $B$, denoted by $A\succ 
    B$, or equivalently  $B$ is \emph{smaller than} $A$, denoted by 
    $B\prec A$, if 
    there exist $A', B'$ such that 
    \begin{center}
        $A'\supset B'$, $A\eq A'$, and $B\eq  B'$.
    \end{center}
\item[\TP]
If $T$ is $1$-$1$, and $T(a)$ is a permutation of $a$ for all $a\in A$ 
then $A\approx T[A]$.
\item[\UP]
$A\* \{n\}\approx A$ for all $n\in\N$.
\item[\PP]
If $A,B$ and $A',B'$ are multipliable pairs and $A\approx A'$, $B\approx B'$ 
then 
\begin{center}
    $A\times B\approx A'\times B'$.
\end{center}

\end{description}

\end{definition}

\begin{definition}\label{num}
{Let $\eq$ be an equinumerosity relation on the set $\WW$.}

\noindent
    $\bullet$ {The \emph{numerosity} of $A$ (w.r.t. $\eq$) is 
    the equivalence class 
$[A]_{\eq}=\{B\in\WW \mid B\approx A\}$ of all point sets
equinumerous to $A$, denoted by 
$\nk_{\approx}(A)$. 
}
   
\noindent
    $\bullet$   {The  \emph{set of numerosities} of 
$\eq$ is the quotient set 
$\Ng_{\eq}=\WW/\eq$,
and}

\noindent
    $\bullet$  {the \emph{numerosity function} 
associated to $\approx$ is
the canonical map $\ng_{\eq}:\WW\to\Nk_{\eq}$.}

{We drop the 
subscript $\eq$ whenever the equinumerosity relation is fixed.}

\end{definition}

\smallskip
Clearly the Unit Principle formalizes the natural
idea that singletons have ``unitary'' numerosity.
A trivial but important consequence of this 
axiom is the existence of infinitely many 
pairwise disjoint equinumerous copies of any point set. Moreover, 
infinitely many of them can be taken multipliable with any fixed set 
of $\WW$. Namely
\begin{proposition}\label{copy}
Let $A,B\in\WW$ be  point sets. For $m,n,h,k\in\N$ put 
$$A(m^{h},n^{k})=A\*\{m\}^{h}\*\{n\}^{k}. $$
Assuming \UP\, the sets $A(m^{h},n^{k})$ are equinumerous to $A$ for all $h,k$, and are 
pairwise disjoint, disjoint from $B$, and multipliable with $B$ for 
all sufficiently large $h,k$.
\end{proposition}
\proof
The first assertion is obvious. In order to prove the remaining ones,
put $p=\max\{m,n\}$ and assume that $A\cap \{0,\ldots,p\}^{l}=\0$ for 
$l\ge j$. Then $A(m^{h},n^{k})$ is disjoint from $A$ and multipliable 
with $B$ whenever $h,k\ge j$. Moreover  $A(m^{h},n^{k})\cap A(m^{h'},n^{k'})=\0$ 
whenever $h,k,h',k'\ge j$ and $k\ne k'$.
\qed

\medskip
In the following proposition we list several important properties of 
the binary relation $\prec$.

\begin{proposition}\label{preorder} ${}$
    
    $(i)$ $A\prec B$ holds if and only if $B$ is equinumerous to a proper 
    superset $B'$ of $A$. Hence, given two sets in $\WW$, one is 
    equinumerous to a superset of the other one. 
    
$(ii)$ The relation $\prec$ is a preorder on $\WW$ that induces a 
total ordering on the quotient set $\Nk=\WW/\eq$.
\end{proposition}
\begin{proof}
$(i)$ If there exists a proper superset $B'$ of $A$ that is 
equinumerous to $B$, then, from the definition 
of $\prec$, we conclude that $A\prec B$.
Conversely,  suppose that $A\prec B$, that is there exist 
sets $A'$ and $B'$ such 
that $A'\subset B'$, $A\approx A'$ and $B\approx B'$.
By possibly applying Proposition \ref{copy}, we may assume without loss 
of generality  that $A\cap A'=A\cap B'=\emptyset$. 
Put $C=B'\setminus A'$, and consider $A\cup C$, which 
 is a proper superset of $A$. By \AP\ we have $$A\cup C\approx B\
\Longleftrightarrow\ A\cup C\approx B'\ \Longleftrightarrow\ A=
(A\cup C)\setminus B'\approx B'\setminus (A\cup C)=A'.$$
So $A\cup C$ is a proper superset of $A$ equinumerous to $B$.

\smallskip
$(ii)$ Clearly $A\not\prec A$ because $A$ cannot be equinumerous to a proper 
superset of itself.
In order to prove transitivity of the relation $\prec$, we state the following lemma:
\begin{lemma}\label{tran}
Let $\approx$ be a relation of equinumerosity on $\WW$, and let $A\approx B$.
Then for each   proper superset $A'$ of $A$
 there exists a proper superset $B'$ of $B$ such that $A'\approx B'$.     
\end{lemma}
\begin{proof}
According to $(i)$, let us consider the three possible cases: 

$(1)$ $A'\approx B$: then $A'\approx A$ against \ZP.

$(2)$ There exists a proper superset $A''$ of $A'$ such that $A''\approx B$: 
then  $A$ would be equinumerous 
to the proper superset $A''$, again contradicting \ZP.

$(3)$ There exists a proper superset $B'$ of $B$ such that $A'\approx B'$ and the lemma
is proved.
\qed
\end{proof}

\medskip
Now we can prove transitivity of the relation $\prec$. Assume that
$A\prec B$ and $B\prec C$: then
there exist proper supersets $A',B'$ of $A,B$ respectively, such that $A'\approx B$ 
and $B'\approx C$. Then, by 
Lemma \ref{tran}, there exists a proper superset $A''$ of $A$ such that 
$A''\approx B'\approx C$ and so 
 $A\prec C$. \qed
\end{proof}

\bigskip

Now we prove that the principle \AP\ is equivalent to the 
conjunction of the second and 
third common notions of Euclid, when formalized in the 
following way:
\begin{description}
 \item [\SP\ (Sum Principle)] Let $A,A',B,B'\in\WW$ be such that $A\cap B=\emptyset$
and $A'\cap B'=\emptyset$. If $A\approx A'$ and $B\approx B'$, then
$A\cup B\approx A'\cup B'$.
\smallskip
 \item [\DP\ (Difference Principle)] Let $A,A',C,C'\in\WW$ be such that
$A\subseteq C$ and $A'\subseteq C'$. If $A\approx A'$ and 
$C\approx C'$, then $C\setminus A\approx C'\setminus A'$.
\end{description}

This equivalence has been proved  in \cite{tesi}, and also in \cite{QSU} for a 
slightly different notion 
of equinumerosity. We repeat the proof here for convenience of the 
reader, because we need these facts in the sequel. 
We begin by stating the following lemma:

\begin{lemma}\label{brutto}
  Let $\eq$ be an equivalence relation for which \AP\ holds and let 
	$A,B,A',B'\in\WW$  be such that                       %
  $B\subseteq A$ and $B'\subseteq A'$.  If $B\approx B'$,
 then                                                                        %
  $$A\setminus B\approx A'\setminus B'\ \ \Iff\ \ A\approx A'.$$     
\end{lemma}
\proof
Put   
  \smallskip\noindent                                                           %
  $B_0=B\setminus A'$,\ \  $B'_0=B'\setminus A$,\ \  $C=B\cap B'$,\ \                     %
  $B_1=B\setminus (B_0\cup C)$,\ \  $B'_1=B'\setminus (B'_0\cup C)$,                  %
  \smallskip\noindent                                                           %
  $C_0=A\setminus(B\cup A')$, $C'_0=A'\setminus (B'\cup A)$,                             %
  $E=(A\cap A')\setminus (B\cup B').$                                                   %
  \smallskip\noindent                                                           %
So we obtain pairwise disjoint sets $B_{0},B'_0,C,B_1,B'_1,C_0,C'_0,E$
such that  
\smallskip\noindent                                                           %
  $B=B_0\cup B_1\cup C$,\ \                                                 %
  $B'=B'_0\cup B'_1\cup C$,\ \                                              %
  \smallskip\noindent                                                           %
  $A\setminus A'=B_0\cup C_0$,\ \                                                  %
  $A'\setminus A=B'_0\cup C'_0$,\ \                                                %
  \smallskip\noindent                                                           %
  $A\setminus B=B'_1\cup C_0\cup E$,\ \                                            %
  $A'\setminus B'=B_1\cup C'_0\cup E$.
\smallskip 
By \AP\, we can write 
$$A\approx A'\ \Longleftrightarrow\ B_0\cup C_0=A\setminus 
A'\approx A'\setminus A=B'_0\cup C'_0,$$
$$ A\setminus B\approx A'\setminus B'\ \Longleftrightarrow\
B'_1\cup C_0=(A\setminus B)
\setminus(A'\setminus B')\approx (A'\setminus B')\setminus(A\setminus B)=B_1\cup C'_0.$$
By hypothesis and \AP\ we can write
$$B_0\cup B_1\cup C=B\approx B'=B'_0\cup B'_1\cup C\ \Longrightarrow\
B_0\cup B_1\approx B'_0\cup B'_1$$
and hence $$B_0\cup B_1\cup C_0\approx B'_0\cup B'_1\cup C_0.$$
Suppose that $A\setminus B\approx A'\setminus B'$, that is $ B'_1\cup C_0\approx B_1\cup C'_0$, 
it follows that $B'_0\cup B'_1\cup C_0\approx B'_0\cup B_1\cup C'_0$, hence 
$B_0\cup B_1\cup C_0\approx B'_0\cup B_1\cup C'_0 $ and we conclude $A\approx A'$.

Conversely, if $A\approx A'$, that is $B_0\cup C_0\approx B'_0\cup C'_0$, we have 
$B_0\cup B_1\cup C_0\approx B'_0\cup B_1\cup C'_0$, hence $B'_0\cup B'_1\cup C_0\approx 
B'_0\cup B_1\cup C'_0$ and we conclude $A\setminus A'\approx A'\setminus A$.

\qed

  \begin{proposition}\label{sumdif}                                             %
   Let $\eq$ be an equivalence relation. The Axiom \AP\ is 
equivalent to the conjunction of the two principles \SP\ and \DP.                
  \end{proposition}                                                             %
\proof
The conjunction of \SP\ and \DP\ yields \AP, namely 
$$(A\setminus A')\cup(A\cap A')=A\approx A'=(A'\setminus A)\cup(A\cap 
A')\
\Longrightarrow^{DP}
(A\setminus A')\approx(A'\setminus A),$$
$$A\setminus(A\cap A')=(A\setminus A')\approx(A'\setminus A)=
A'\setminus(A\cap A')\ \Longrightarrow^{SP}
A\approx A'.$$
Conversely, assume \AP: then, by Lemma \ref{brutto}, both \SP\ and \DP\ hold.
\qed

\medskip
We  prove now that our notion of equinumerosity satisfies what can be
viewed as a necessary condition, \emph{videlicet}
that \emph{finite point-sets are equinumerous if and only if they 
have the same 
``number of elements''}. 
\begin{proposition}\label{fin}
	Let $\eq$ be an equinumerosity relation, and let 
	$A,B\in\WW$ be finite sets. Then
	$$A\approx B\ \Iff\ |A|=|B|.$$
       Moreover, if $X$ is infinite, then $X\succ A$.
       Hence $\N$ can be taken as an initial segment of the set of
      numerosities $\Nk$ corresponding to $\eq$.
\end{proposition}

{\bf Proof.}
First observe that $\emptyset$, being a proper subset of any nonempty 
set $A$, cannot  be equinumerous to $A$.

Secondly, we have already remarked that any two singletons $\{a\},\{b\}$ 
are equinumerous.
Moreover, if $C\eq \{b\}$, then $C$ is a singleton. In fact
let $c$ be an element of $C$; then 
$\{c\}\approx  \{b\}\approx C$, hence $\{c\}=C$, because $C$ cannot 
be a proper superset of  \{c\}.

Finally, given two finite sets $A$ and $B$, we proceed by induction 
on $k$, the least cardinality of the sets $A,B$.
The case $k=1$ has already been dealt with. 
Assume the thesis  true for $k\leq n$ and let $A$ and $B$ be 
finite sets such 
that $n+1=\min\{|A|,|B|\}$.
Pick $a\in A$
and $b\in B$, and put $A'=A\7\{a\}$, $B'=B\7\{b\}$. Since 
$\{a\}\eq\{b\}$, Lemma \ref{brutto} and the induction hypothesis yield
$$A\eq B\ \Iff\ A'\eq B'\ \Iff\ |A'|=|B'|\ \Iff\ |A|=|B|.$$ 

Now if $X$ is an infinite set and $A$ is a finite set, we can find a proper subset $B$ of $X$, such that 
 $|A|=|B|$, so we conclude that $A\prec X$. 
 
\qed

By the above proposition, we can identify
each natural number $n\in\N$, with the
equivalence class of all those point sets
that have finite cardinality $n$.

\medskip
Starting from the equivalence relation of \emph{equipotency},
Cantor introduced the algebra of cardinals by means
of disjoint unions and Cartesian products.
So we could similarly introduce an algebra on ``numerosities''.
The given axioms have been chosen so as to guarantee that numerosities are naturally equipped
with a ``nice'' algebraic structure.
(This is to be contrasted with the awkward cardinal algebra,
where \emph{e.g.}
$\kappa+\mu=\kappa\cdot\mu=\max\{\kappa,\mu\}$ for all infinite
$\kappa,\mu$.)

\begin{theorem}\label{ring}
    Let $\Nk$ be the set of numerosities of the equinumerosity 
    relation $\eq$. Then
there exist unique operations $+$ and $\cdot$,
and a unique linear order $<$ on $\Ng$, such that for
all point sets $X,Y\in\WW$:
\begin{enumerate}
\item
$\ng(X)+\ng(Y)=\ng(X\cup Y)$ whenever $X\cap Y=\emptyset$\,;
\item
$\ng(X)\cdot\ng(Y)=\ng(X\times Y)$ whenever $X, Y$ are multipliable\,; 

\item
$\ng(X)<\ng(Y)$ if and only if $Y\approx Y'$
for some proper superset $Y'\supset X$.
\end{enumerate}

The resulting structure on $\Ng$
is the non-negative part of a discretely ordered ring
$(\mathfrak{R}, 0, 1, +, \cdot, <)$.
Moreover, if the fundamental subring of $\Rg$ is identified with 
$\Z$, then
$\ng(X)=|X|$ for every finite point set $X$.
\end{theorem}

We could prove the above theorem by the very same arguments used in 
\cite{QSU} or in \cite{tesi}. However we prefer to obtain a more 
precise algebraic characterization of the arithmetic of numerosities. 
To this aim, we consider a suitable ring of formal power 
series with 
integer coefficients, and we prove that the set of numerosities can be identified 
with
the non-negative part of the 
quotient of this ring   modulo a suitable prime ideal.

\noindent
$\bullet$ Let $\TT=\la t_{n}\mid\, n\in\N\ra$ be a sequence of 
 indeterminates.  
 Let $\Ab$ be the set of all eventually zero 
 sequences
 $\aa=(a_0,a_1,\ldots)$ of non-negative integers, and 
denote by $t^{\aa}$ the monomial 
$\prod_{i\in\N} t_i^{a_i}.$

So any series  $S$ in the variables of $\TT$ can be written as
$S=\sum_{\aa\in\Ab} n_{\aa}t^{\aa}$ {where}\,\,$
n_{\aa} $
{is the  coefficent of the monomial}\ $t^{\aa}$.

    \noindent
$\bullet$   Given a point
$x=(x_{1},\ldots,x_{d})\in\N^{d}$, consider the sequence $\aa\in\Ab$, 
where $a_i=|\{j\mid x_j=i\}|$ and associate to $x$ the \emph{monomial}  
$t_{x}=t^{\aa}.$

\noindent
$\bullet$ The \emph{characteristic series}
of the nonempty point set $X\in\WW$ is the
 formal series
$$S_{X}=\sum_{x\in X}t_{x}=\sum_{\aa\in\Ab}n_{\aa}t^{\aa},\
\mbox{where}\ \ n_{\aa}=|\{x\in X\mid t_{x}=t^{\aa}\}|.$$ 
(If $X=\emptyset$,  put $S_X=0$.)
  
\noindent
$\bullet$ Characteristic series behave well with respect to
 \emph{unions}, \emph{differences} and  \emph{products}:
$$\ \ S_{X}+S_{Y}=S_{X\cup Y}+S_{X\cap Y}\ \ \mbox{,}\ \
\ S_X-S_Y=S_{X\setminus Y}\ \mbox{if}\ Y\subset X$$
$$S_{X}\cdot S_{Y}=S_{X\times Y}\ \mbox{if}\ X,Y\ \mbox{are 
multipliable}.\footnote{~Remark that, if the product $X\* Y$ is 
considered as a multiset where each tuple comes with its 
multiplicity, then the equality holds for all $X,Y\in\WW$.}$$

\begin{remark}\label{condizione}
A series $S=\sum n_{\aa}t^{\aa}$ with \emph{non-negative} integer 
coefficients  is the characteristic 
series of a set $X\in\WW$ if and only if for all $\aa=(a_0,a_1,\ldots)$ we have 
$n_{\aa}\leq\frac{k!}{\prod a_i!}$, where $k=\sum_{i}a_{i}$. 
\end{remark}
 In  fact, the number of 
different sequences 
$\aa=(a_0,a_1,\ldots)$ that correspond to the same monomial $t^{\aa}$ is 
$\frac{k!}{\prod a_i!}$, where $k=\sum_{i}a_{i}$ is the 
degree of the monomial.

\medskip 
\noindent
$\bullet$ Let $\RR$ be  the ring of
 all formal series of \emph{bounded degree $d_{n}$ in each variable} 
 $t_{n}\in\TT$
with coefficients $n_{\aa}$ such that, for some 
$ b\in\N$,
$$|n_{\aa}|\le
b\frac{(\sum_{i}a_{i})!}{\prod a_i!} . $$

\noindent
$\bullet$ Let $\RP$ be the
multiplicative subset of the \emph{positive  series}, \ie\ 
the series  in $\RR$ having only positive 
coefficients.

\noindent
$\bullet$ Let $\Ik_{0}$ be the ideal of $\RR$
generated by $\{t_{n}-1\,\mid\, n\in\N\}$.

\medskip
It is easily seen that 
$\RR$ is the subring with 
identity of $\Z[[\TT]]$ generated by
   the set of the characteristic series of all point sets. Moreover
every positive series $P\in\RP$ is equivalent modulo $\Ik_{0}$ 
to some characteristic series.

\begin{lemma}\label{equivalenza}
 Every positive series $P\in\RP$ is equivalent modulo $\Ik_{0}$ 
to the characteristic series of some set in $\WW$. So any 
series  $S\in\RR$ can be written as $S=S_{X}-S_{Y}+S_{0}$ for suitable 
$X,Y\in\WW$ and $S_{0}\in\Ik_{0}$. 
\end{lemma}

\begin{proof}
Suppose that the series $S$ has  non-negative
coefficients $n_{\aa}\le 
b\frac{k!}{\prod a_i!}$, where $k=\sum_{i}a_{i}$. 
We can decompose $S$ in at most $b$  series with coefficients 
satisfying the conditions for being  characteristic, 
plus a non-negative integer $a$. So we can write 
$S=a+S_{X_1}+\ldots+S_{X_s}$, where 
$s\leq b$ and for all $i$, $X_i\in\WW$. (The sets $X_{i}$ may be not 
distinct, in general.)
According to Proposition \ref{copy},
we can multiply the integer $a$ and each series $S_{X_i}$ by 
 suitable monomials 
$t_m^{h}t_{n}^{k}$,  so that 
$at_m^{h}t_{n}^{k}$ is the characteristic series of a set $Y_{0}$ of 
tuples containing $h$ times $m$ and $k$ times $n$, and
the remaining series are the characteristic 
series of pairwise disjoint sets $Y_{i}$ equinumerous to $X_{i}$.
Clearly, each $S_{X_i}$ is equivalent modulo $\Ik_{0}$ to 
 $S_{Y_i}$, 
  so putting $Y=Y_{0}\cup Y_1\cup\ldots\cup Y_s$, we have 
that $S$ is equivalent modulo $\Ik_0$ to the characteristic 
series $S_Y=S_{Y_{0}}+\ldots+S_{Y_{s}}$.

The final assertion of the theorem follows by considering separately 
the positive and the negative parts of any given series of $\RR$.
\qed
\end{proof}

\medskip
In order to classify all equinumerosities, we introduce the following 
definition:
\begin{definition}
    Call  an ideal $\Ik$ of $\RR$ a
   {\emph{gauge ideal}} if 
   \begin{itemize}
       \item  $\Ik_{0}\incl \Ik$,

       \item  $\RP\cap\Ik=\0$, and

       \item for all $S\in\RR\7 \Ik$ there exists  
   $P\in\RP$ such that either $\ S+ P\in\Ik$ or $\ S- P\in\Ik$. 
\end{itemize}  
\end{definition}

Remark that when $\Ik$ is a gauge ideal the quotient $\RR/\Ik$  is a 
\emph{discretely ordered  ring} whose 
\emph{positive elements} are the cosets $P+\Ik$ for $P\in\RP$. 
In particular $\Ik$ is a \emph{prime} ideal of $\RR$ that is 
\emph{maximal} among the ideals disjoint from $\RP$.

Then we have

   \begin{theorem}\label{ser}
       There exists a biunique correspondence between 
       equinumerosity relations on the
       space $\WW$ of all point sets over $\N$ and gauge ideals on 
       the ring $\RR$ of all bounded power series in countably many 
        indeterminates. 
       In 
       this correspondence, if the equinumerosity $\eq$ corresponds to the 
       ideal  $\Ik$,  then
       \begin{equation}
    X\eq Y\ \Iff\ \ S_{X}-S_{Y}\in\Ik. \tag{**}
    \label{**}
\end{equation}

       More precisely, let $\nk:\WW\to\Nk$ be the numerosity function associated 
       to $\eq$,
       and let $\pi:\RR\to\RR/\Ik$ be the canonical projection. Then
   there exists a unique order preserving embedding $j$ of $\Nk$ onto 
   the non-negative part of $\RR/\Ik$
   such that the following diagram commutes
\bigskip
\begin{center}
\begin{picture}(90,70)
   \put(0,0){\makebox(0,0){$\Nk$}}
   \put(90,0){\makebox(0,0){$\RR/\Ik$}}
   \put(0,70){\makebox(0,0){$\WW$}}
   \put(45,35){\makebox(0,0){\diser}}
   \put(90,70){\makebox(0,0){$\RR$}}
   \put(45,6){\makebox(0,0){$j$}}
   \put(45,76){\makebox(0,0){$\Sigma$}}
   \put(-8,35){\makebox(0,0){$\nk$}}
   \put(100,35){\makebox(0,0){$\pi$}}
   \put(15,0){\vector(1,0){60}}
   \put(15,70){\vector(1,0){60}}
   \put(0,60){\vector(0,-1){50}}
   \put(90,60){\vector(0,-1){50}}
\end{picture}
\end{center}

	   \bigskip

\noindent
$($where $\Sg$ maps any $X\in\WW$ to its characteristic series 
$S_{X}\in\RP.)$

 The ordered semiring structure induced on $\Nk$ by $j$ satisfies the 
 conditions 
 \begin{enumerate}
\item
$\ng(X)+\ng(Y)=\ng(X\cup Y)$ whenever $X\cap Y=\emptyset$\,;

\item
$\ng(X)\cdot\ng(Y)=\ng(X\times Y)$\, whenever $X,Y$
are multipliable;
\item
$\ng(X)<\ng(Y)$ if and only if $Y\approx Y'$
for some proper superset $Y'\supset X$.
\end{enumerate}

	     \end{theorem}
	     
\proof

Given a gauge ideal $\Ik$ over $\RR$, we  define the equivalence relation 
over $\WW$ by condition (\ref{**}): $X\eq Y\ \Iff\ \ S_{X}-S_{Y}\in\Ik.$
We prove that the relation $\eq$ so defined is in fact an equinumerosity
relation.
\begin{description}
 \item[\AP] We can write 
 $S_X-S_Y=S_{X\setminus Y}+S_{X\cap Y}-S_{Y\setminus X}-S_{X\cap Y}
=S_{X\setminus Y}-S_{Y\setminus X}$, hence 
$X\eq Y\ \Iff\ X\setminus Y\eq Y\setminus X$.

\smallskip\noindent
\item[\ZP] Assume first  that $S_X-S_Y\in\Ik$, that is $X\eq Y$. 
If there exists a superset $Y'$ of $X$ 
such that $S_{Y'}-S_Y\in\Ik$, then 
$S_{Y'}-S_X=S_{Y'\setminus X}\in\Ik\cap\RP$, contradiction. Similarly there 
exists no superset $X'$ of $Y$ equinumerous to $X$.

Now  assume that $S_X-S_Y\not\in\Ik$, that is $X\not\eq Y$; then 
there exists a series $P\in\RP$ such that either $S_X-S_Y+P$ or $S_X-S_Y-P$ 
belongs to $\Ik$, by the third property of  gauge ideals.
Assume without loss of generality that $S_X-S_Y+P\in\Ik$. Then
 $P$ is 
equivalent modulo $\Ik$ to the characteristic series of a set $Z$ 
that we can choose disjoint from $X$, by Proposition \ref{equivalenza}.
Hence 
 $S_{X\cup Z}-S_Y=S_X-S_Y+S_Z\in\Ik$, whence $(X\cup Z)\eq Y$. 
 
 \smallskip\noindent
\item[\TP]  $S_X$ is obviously equal to $S_{T[X]}$ whenever
 $T$ is a transformation that permutes the components of the tuples in $X$.
 Hence $X\eq T[X]$.
 
 \smallskip\noindent
\item[\UP] For any set $X\in\WW$, we have
$S_{X\times\{n\}}-S_X=S_X\cdot(t_{n}-1)\in\Ik_{0}\incl\Ik$, \ie\ 
$X\times\{n\}\eq X$.

\smallskip\noindent
\item[\PP] Let $X,Y$ and $X',Y'$ be multipliable point sets.
 If $S_X-S_{X'}$ and 
$S_Y-S_{Y'}$ belong to $\Ik$, that is $X\eq X'$ and $Y\eq Y'$, then 
$$S_{X\times Y}-S_{X'\times Y'}=
(S_{X}-S_{X'})\cdot S_{Y}+S_{X'}\cdot (S_{Y}-S_{Y'})\in\Ik,$$  hence 
$X\times Y\eq X'\* Y'$.
\end{description}
Conversely, given an equinumerosity relation $\eq$ over $\WW$, let 
$\Ik$ be the ideal of $\RR$ generated by the 
set $\{S_X-S_Y\mid X\eq Y\}\cup\{t_0-1\}$. We prove that $\Ik$
is a gauge ideal.

First observe that the ideal $\Ik_0$ is also generated by 
$\{t_n-t_0\mid n\in\N\}\cup\{t_0-1\}$,
  so  $\Ik_0$ is contained in $\Ik$, because 
 $t_n-t_0=S_{\{n\}}-S_{\{0\}}$.
 
 By Proposition \ref{equivalenza}, given a series $S\in\RR$, there exist two 
 sets $X$ and $Y$
of $\WW$, such that $S$ is equivalent modulo $\Ik$ to $S_X-S_Y$.
If $S\notin\Ik$, then $X\not\eq Y$, so, without loss of generality,
we may assume that there exists a superset 
$X'$ of $Y$ such that $X\eq X'$,  that is $S_X-S_{X'}\in\Ik$,  by \ZP. 
Then we can write 
$S_X-S_{X'}=S_X-S_{X'\setminus Y}-S_{Y}\in\Ik$ and hence 
$S_X-S_{Y}$ is congruent modulo $\Ik$ to $S_{X'\setminus Y}\in\RP$.

It remains to prove that $\RP\cap\Ik=\emptyset$. We need the 
following fact:

\begin{claim}
For any $S\in\Ik$ there exist $h,k\in\N$ and finitely many (not necessarily 
distinct) sets  $Z_1$, $W_1$, 
$\ldots$, $Z_m$, $W_m$ of $\WW$, such that 
$Z_1\eq W_1$, $\ldots$, $Z_m\eq W_m$ and 
$$t_{0}^{h}t_{1}^{k}S=(S_{Z_1}-S_{W_1})+\ldots+
(S_{Z_m}-S_{W_m}).$$
\end{claim}
\proof
By definition of $\Ik$, there exist $S_0$, $\ldots$, $S_n\in\RR$ such that 
$$S=S_0\cdot (t_0-1)+S_1\cdot(S_{X_1}-S_{Y_1})
+\ldots+S_n\cdot(S_{X_n}-S_{Y_n}),$$
with $X_1\eq Y_1$, $\ldots$, $X_n\eq Y_n$.

According to 
Proposition \ref{equivalenza}, there exist $a_{i}\in\Z$ and 
$X_{1i},\ldots,X_{si};Y_{1i},\ldots,$ $Y_{ti}\in\WW$ such that
$$S_i=a_i+S_{X_{1i}}+\ldots+S_{X_{si}}-S_{Y_{1i}}-\ldots-S_{Y_{ti}}.$$
Then we have 
\begin{equation*}
\begin{split}
S_i\cdot(S_{X_i}-S_{Y_i}) &=a_i\cdot(S_{X_i}-S_{Y_i})+S_{X_{1i}}
\cdot(S_{X_i}-S_{Y_i})+\ldots+S_{X_{si}}\cdot(S_{X_i}-S_{Y_i})\\ 
                          &\ \ \ -S_{Y_{1i}}\cdot(S_{X_i}-S_{Y_i})-\ldots
			  -S_{Y_{ti}}\cdot(S_{X_i}-S_{Y_i}).  
                          \end{split}
\end{equation*}

By Proposition \ref{copy} we can find $h,k\in\N$ such that all sets 
$$U_{ji}=X_{ji}\*\{0\}^{h}\*\{1\}^{k},V_{li}=Y_{li}\*\{0\}^{h}\*\{1\}^{k}$$ 
are multipliable with both $X_{i}$ and $Y_{i}$.

For all $i>0$  we have 
\begin{equation*}
\begin{split}
t_{0}^{h}t_{1}^{k}S_i(S_{X_i}-S_{Y_i})=& \ a_it_{0}^{h}t_{1}^{k}(S_{X_i}-S_{Y_i})+S_{U_{1i}}(S_{X_i}-S_{Y_i})+
\ldots+S_{U_{si}}(S_{X_i}-S_{Y_i})\\
& -S_{V_{1i}}(S_{X_i}-S_{Y_i})-\ldots-S_{V_{ti}}(S_{X_i}-S_{Y_i})\\
=& \ a_i(S_{\{0\}^{h}\*\{1\}^{k}\*X_i}-S_{\{0\}^{h}\*\{1\}^{k}\*Y_i})+
 (S_{U_{1i}\times X_i}-S_{U_{1i}\times Y_i})\\ 
&+\ldots+(S_{U_{si}\times X_i} 
 -S_{U_{si}\times Y_i})-(S_{V_{1i}\times X_i}-S_{V_{1i}\times Y_i})-
 \ldots\\
&-(S_{V_{ti}\times X_i}-S_{V_{ti}\times Y_i}).
\end{split}
\end{equation*}
Similarly, observing that 
$t_0-1=S_{\{0\}}-1$ we obtain for $i=0$\\
\begin{equation*}
\begin{split}
t_{0}^{h}t_{1}^{k}S_0(t_{0}-1)=&
a_0(S_{\{0\}^{h+1}\*\{1\}^{k}}-S_{\{0\}^{h}\*\{1\}^{k}})+
 (S_{U_{10}\times \{0\}}-S_{U_{10}})+
 \ldots\\ 
&+(S_{U_{s0}\times \{0\}} 
  -S_{U_{s0}})-(S_{V_{10}\times \{0\}}-S_{V_{10}})-
 \ldots-(S_{V_{t0}\times \{0\}}-S_{V_{t0}}).
\end{split}
\end{equation*}
 (If $a_{0}>0$ we take $a_{0}$ times the difference 
 $S_{\{0\}^{h+1}\*\{1\}^{k}}-S_{\{0\}^{h}\*\{1\}^{k}}$, whereas if 
 $a_{0}<0$ we take $-a_{0}$ times the difference 
 $S_{\{0\}^{h}\*\{1\}^{k}}-S_{\{0\}^{h+1}\*\{1\}^{k}}$.)

\smallskip\noindent
By \PP\ we have $U_{ji}\times X_i\eq U_{ji}\times Y_i$
and $V_{ji}\times X_i\eq V_{ji}\times Y_i$ for all $j$, and obviously 
$\{0\}^{h}\*\{1\}^{k}\*X_i\eq \{0\}^{h}\*\{1\}^{k}\*Y_i$. 
Hence, by taking $Z_{1},\ldots,Z_{m}$ and $W_{1},\ldots,W_{m}$
to be an enumeration of the sets $U_{ji}\times X_i,V_{ji}\times X_i,
{\{0\}^{h+1}\*\{1\}^{k}}$ 
and $U_{ji}\times Y_i,V_{ji}\times Y_i,{\{0\}^{h}\*\{1\}^{k}}$, respectively,
we obtain the claim. 
\qed

\bigskip
Clearly, $t_{0}^{h}t_{1}^{k}S\in\RP$ if and only if $S\in\RP$.
Towards a contradiction, assume that $S$ be in $\Ik\cap\RP$.
By the Claim, there exist $h,k\in\N$ such that  
 $t_{0}^{h}t_{1}^{k}S=(S_{Z_1}-S_{W_1})+\ldots+
(S_{Z_m}-S_{W_m}),$ for suitable $Z_{i}\eq W_{i}$.

Since $S$ belongs to $\RP$,  each $W_i$ can be 
decomposed into pairwise disjoint subsets: $W_i=\bigcup_{j=1}^m W_{ij}$, 
in such a way that 
 $W_{ij}\incl Z_j$.
Put $P=(1,\ldots,m)$ and let $\sg_{1},\ldots,\sg_{m}$ be different 
permutations of $\{1,\ldots,m\}$. Put $X_{i}=\{\sg_{i}(P)\}\* Z_{i}$
and $Y_{i}=\{\sg_{i}(P)\}\* W_{i}$. Then we have 
$t_{P}\cdot S_{Z_i}=S_{X_i}$ and 
$t_{P}\cdot S_{W_i}=S_{Y_i}$, where $X_1,\ldots,X_m,
$ and $Y_1,\ldots,Y_{m},$ are pairwise disjoint sets of $\WW$. 

Put $X=\bigcup_{i=1}^mX_i$ and $Y=\bigcup_{i=1}^m Y_i$.
Then the series  $$t_{P}t_{0}^{h}t_{1}^{k} S=
(S_{X_1}-S_{Y_1})+\ldots+(S_{X_m}-S_{Y_m})=(S_X-S_Y)$$
is still in
$\RP\cap\Ik$.
 By \UP\ 
we have $Z_i\approx X_i\approx Y_i\approx W_i$,  hence $X\approx Y$, by \SP.

For each monomial $t^\bb$ of the series $S_Y$ there
exists a point $Q\in Y_i$ such that $t^\bb=t_{Q}$. Note that 
$t_{Q}=t_{P}t_{Q_{i}}$, where $Q_{i}$ is a point in ${W_i}$, so,
by hypothesis, $t_{Q_{i}}$ is also a monomial of some series $S_{Z_j}$.
Hence there exists a permutation of the coordinates of $Q$ such that the 
corresponding point $Q'$
belongs to $X_j\subset X$.
Repeating this argument for each monomial of the series $S_Y$, we 
can conclude that there exists a transformation $T$ of $Y$, such that 
$T[Y]\subseteq X$ and $T[Y]\approx Y\approx X$, hence $X=T[Y]$.
Then we have $S_X-S_Y=S_X-S_{T[Y]}=0$ and so $(S_{X_1}-S_{Y_1})+
\ldots+(S_{X_m}-S_{Y_m})=0$, 
and this implies $S=0$, absurd. So the proof that $\RP\cap\Ik=\0$ is 
complete.

\smallskip
Given an equinumerosity relation $\eq$, let 
$\varphi(\approx)$ be tha gauge ideal generated by the set $\{S_X-S_Y\mid
 X\approx Y\}\cup\{t_0-1\}$.
 The map
$$\varphi\colon\{\eq\,\mid\ \eq\, \mbox{equinumerosity relation of}\ \WW\}
\rightarrow\{\Ik\,\mid\,\Ik\,\mbox{gauge ideal of}\ \RR\}$$
 is iniective: in fact let $\approx_1$
 and $\approx_2$ be two different equinumerosity relations, then
there exist two sets $X,Y$ of $\WW$, such that $X\approx_1 Y$ 
and $X\not\approx_2 Y$. Assume without loss of generality 
that there exists a proper superset $Y'$ of $X$ such that 
$Y'\approx_2 Y$. Put $\varphi(\approx_1)=\Ik_1$ and 
$\varphi(\approx_2)=\Ik_2$: then we have $S_X-S_Y\in\Ik_1$ and 
$S_{Y'}-S_Y\in\Ik_2$. If $S_X-S_Y$ belongs to $\Ik_2$, then 
the series $S_{Y'}-S_{X}=S_{Y'\setminus X}$ is an element of 
$\RP\cap\Ik_2$, absurd; so $\Ik_1\neq\Ik_2$ and $\fg$ is $1$-to-$1$.

Given a gauge ideal $\Ik$, consider the equinumerosity 
relation 
$\approx$ defined by $\Ik$ through (\ref{**}). Putting $\varphi(\approx)=
\Ik_{\approx}$, we have obviously $\Ik_{\approx}\subseteq\Ik$, 
but, if  there exists 
$S\in\Ik\setminus\Ik_{\approx}$, then there exists $P\in\RP$ such 
that either $S+P$ or $S-P$ belongs to $\Ik_{\approx}$. Without loss of generality
suppose that $S+P\in\Ik_{\approx}$: then $S+P\in\Ik$ and so $P\in\Ik$,
 absurd. Therefore $\Ik_{\approx}=\Ik$ and  $\varphi$ is biunique.

Fixed an equinumerosity relation $\approx$ and the corrisponding gauge 
ideal $\Ik$,  define  $j:\Nk\rightarrow\RR/\Ik$ 
by $j(\nk(X))=S_X+\Ik$.
The application $j$ is well-defined, because, whenever $X,Y$ are two equinumerous 
sets of $\WW$,  we have $S_X-S_Y\in\Ik$,  hence 
$j(\nk(X))=j(\nk(Y))$.
Moreover, if $j(\nk(X))=j(\nk(Y))$, 
that is $S_X-S_Y\in\Ik$, then $\nk(X)=\nk(Y)$, hence 
$j$ is a embedding of $\Nk$ into the non-negative part of $\RR/\Ik$.
The range of $j$ is exactly $(\RP+\Ik)/\Ik$ because each element of 
$\RP$ is congruent modulo $\Ik$ to a characteristic series.

We prove that $j$ induces an ordered semiring structure on $\Nk$. 
Given two disjoint sets $X,Y$ of $\WW$,  define 
$\nk(X)+ \nk(Y)=j^{-1}(j(\nk(X)+j(\nk(Y)))$: 
then  $$\nk(X)+ \nk(Y)=j^{-1}(S_X+S_Y+\Ik)=j^{-1}(S_{X\cup Y}+\Ik)=
\nk(X\cup Y).$$
Similarly, given two multipliable sets $X,Y$ of $\WW$,  define 
$\nk(X)\cdot \nk(Y)=j^{-1}(j(\nk(X)\cdot j(\nk(Y)))$: 
then  $$\nk(X)\cdot \nk(Y)=j^{-1}(S_X\cdot S_Y+\Ik)=j^{-1}(S_{X\times Y}+\Ik)=
\nk(X\times Y).$$
Sum and product are well defined by  Proposition 
\ref{copy}.

 If $X,Y\in\WW$ are  not equinumerous, \ie\
 $S_X-S_Y\in\RR\setminus\Ik$, then there 
exists a set $Z\in\WW$ such that either 
$S_X-S_Y+S_Z$ or $S_X-S_Y-S_Z$ belongs to $\Ik$, 
by the third propriety of gauge ideals.
So define a total order $<$ on $\Nk$ by putting
$$\nk(X)<\nk(Y)\ \Iff\  \exists Z\in\WW,\, Z\ne\0\ \mbox{such that}\ 
    S_X-S_Y+S_Z\in\Ik.$$

Note that, by  Proposition \ref{equivalenza}, if $\nk(X)<\nk(Y)$ we can 
always choose the set 
$Z$ so that $X$ and $Z$ are disjoint and hence 
$S_X-S_Y+S_Z=S_{X\cup Z}-S_Y\in\Ik$.
Therefore $\nk(X)<\nk(Y)$ holds if and only if there exists a proper superset 
$Y'$ of $X$ such that $Y'\approx Y$. By  Proposition \ref{preorder} we 
conclude that the relation $<$ is a total order on $\Nk$ and hence 
$(\Nk,+,\cdot,<)$ is a ordered semiring.
\qed

\begin{remark}
    In \cite{tesi} a class of particular equinumerosity relations has 
    been considered, namely those which are preserved under 
    ``natural transformations'', \ie\  bijections that preserve the
    \emph{support} (set of 
components) of each tuple. 
   
   Let us call  an equinumerosity relation \emph{natural} if it 
   satisfies the following Natural Transformation Principle, which is 
   a severe strengthening of \TP:
   \begin{itemize}     
         \item[(\NP)] If $T$ is 1-to-1 on $X\in\WW$ and $supp(T(x))=supp(x)$ 
	 for all $x\in X$, then
	 {$X\eq T[X]$. }
    \end{itemize}
   
   For $S\in\RR$ let $S'$ be the  \emph{associated squarefree 
   series}, \ie\ the series obtained by replacing each monomial in 
   $S$ by the corresponding squarefree monomial and summing up the 
   corresponding coefficients. \Ie 
   $$\mbox{if}\ \  S=\sum_{\aa\in\lA}n_{\aa}t^{\aa}\ \ \mbox{then}\ \
   S'=\sum_{F\in\N^{<\og}}(\sum_{supp(\aa)=F}n_{\aa})\,t_{F}$$
   where 
   \begin{center}
       $supp(\aa)=\{n\in\N\mid a_{n}\ne 0\}\ $ and\ $\ t_{F}=\prod_{n\in 
          F}t_{n}$.
   \end{center}
   
\smallskip
Let  $\Ik_{1}$ be the kernel of the map $S\mapsto S'$, \ie\ the ideal of $\RR$ 
   generated by the set $\{S-S'\,\mid\,S\in\RP\}$.
   Then we have
   
   \begin{proposition}
       The equinumerosity relation $\eq$ is natural if and only if 
       the corresponding gauge ideal $\Ik$ includes $\Ik_{1}$.
   \end{proposition}
\proof
Let $\Ik$ be a gauge ideal that includes $\Ik_1$. We prove that 
the equivalence relation defined by the relation
\begin{equation*}
X\approx Y \Iff\ S_X-S_Y\in\Ik 
\end{equation*} 
is a natural equinumerosity.
The principles \AP, \ZP, \UP, \PP\ hold because the ideal $\Ik$ is 
gauge, hence we have only to prove that if $T$ is a natural transformation of 
a set $X\in\WW$ 
then $X\approx T[X]$, or equivalently $S_{X}-S_{T[X]}\in\Ik$.

Let $S'_X$ and $S'_{T[X]}$ be 
the  squarefree series associated to the series $S_X$ and $S_{T[X]}$, 
respectively. We prove that $S'_X=S'_{T[X]}$, so
 the series $S_{X}-S_{T[X]}=S_X-S'_X+S'_{T[X]}-S_{T[X]}$ belongs 
to $\Ik_{1}\incl\Ik$.
Let $n_\aa t^\aa$ be a (squarefree) monomial of the series $S'_{X}$, 
then there exist distinct points 
$P_1,\ldots,P_{n_\aa}$ of the set $X$ that produce the monomial $t^\aa$. 
Now consider the points $T(P_1),\ldots,T(P_{n_\aa})$, that belong to the set $T[X]$. 
The map $T$ preserves the support of each point, hence in the series $S'_{T[X]}$ 
the monomial $t^\aa$ appears with a coefficent equal to $n_\aa$ 
because all and only $T(P_1),\ldots,T(P_{n_\aa})$ can have this 
support. So
we  conclude that $S'_{T[X]}=S'_{X}$.

Conversely, given a natural equinumerosity $\approx$ over $\WW$, let $\Ik$ be 
the ideal generated by the set $\{S_X-S_Y\mid X\approx Y\}\cup\{t_0-1\}$. 
The ideal $\Ik$ is gauge, because $\approx$ is an equinumerosity relation, 
so we need only to prove that 
$\Ik$ includes the ideal $\Ik_1$ generated by 
the set $\{S-S'\,\mid\,S\in\RP\}$.

Let $S$ be a series of $\RP$, and pick (not necessarily distinct) sets
$X_1,\ldots,X_n\in\WW$ such that 
$S=a+S_{X_1}+\ldots+S_{X_n}$, taking care that, in each set  $X_i$, 
the tuples of each support of size $k$ be at 
most $k!$. 
Then each squarefree series $S'_{X_i}$ is equal to the characteristic series of 
a set $Y_i\in\WW$, and we can define  natural tranformations 
$T_i\colon X_i\rightarrow  Y_i$ in such a way that
  $T_i(x)$ is a permutation of the support of $x$, for all $x\in X_i$.
  Hence  
$$S'=a+S'_{X_1}+\ldots+S'_{X_n}=a+S_{Y_1}+\ldots+S_{Y_n}=
a+S_{T_{1}[X_1]}+\ldots+S_{T_{n}[X_n]},$$
 and we  conclude that the series 
$S-S'$ belongs to $\Ik$.

\qed

\medskip
The very same proof of Theorem \ref{ser} can be used to
prove that there exists a biunique corrispondence between \emph{natural 
equinumerosities} and 
\emph{gauge ideals $\Ik$ including} $\Ik_1$, and that there exists 
a unique order preserving embedding $j$ of the set $\Nk$ of the
natural numerosities 
 corrisponding to  the ideal $\Ik$ onto 
the non-negative part of $\RR/\Ik$. In the next section we shall give a 
complete characterization of all natural numerosities as 
\emph{hypernatural numbers} 
from suitable ultrapowers of $\N$.
   
 \end{remark}  
   
   \section{Equinumerosities  through ultrafilters}\label{ult}
We give in this section a construction of equinumerosity relations 
through suitable ultrafilters.

\begin{definition}\label{seq}
    {Let $\XX$ be the set of all sequences  
of non-negative real numbers $\xb=\la x_{0},\ldots,x_{n},\ldots\ra$ 
such that the series $\sum x_{n}=\|\xb\|$ converges.\\
For $S\in \RR$\ let
$ S(\xb)$ be the value taken by $S$ when $x_{n}$ is assigned 
to the variable $t_{n}$. (So \pes\ $\|\xb\|^{d}=S_{\N^{d}}(\xb).)$}
   \begin{itemize}
 
         \item  {A subset $\II\incl \XX$ is
  a  \emph{counting set} (of assignements) if } 
      {for all $k\in\N$ the set}\ \   
   $\II_{k}= \{\ib\in\II \mid  i_{0}=i_{1}=\ldots =i_{k}=1\ \}\ne\0;$
         
   \smallskip
         \item {an ultrafilter $\U$ on the counting set $\II$ is 
        \emph{suitable for} $\II$ 
 if for all $k\in\N$ the set $\II_{k}$ is in $\U$;} 
   \smallskip
\item {the counting set $\II$  is \emph{suitable
for} the ideal $\Ik$ of $\RR$
if for all $S\in\Ik$ and all $k\in\N$ there exists 
$\ib\in \II_{k}$ such that $S(\ib)=0$.}

   \end{itemize} 
  
\end{definition}

Let $\II$ be a counting set,
and 
define the \emph{counting map} $$\Phi:\RR\to \R\ult{\II}{}\ \
\mbox{by}\ \
 \Phi(S)=\la S(\ib)\ra_{\ib\in \II}. $$

Then $\Phi$ is a ring homomorphism that preserves the respective partial 
orderings. 

If $\U$ is an ultrafilter suitable for $\II$ and 
$\pi_{\U}:\R^{\II}\to\R\ult{\II}{\U}$ is the 
natural projection onto the corresponding ultrapower, put 
$\phi_{\U}=\pi_{\U}\circ\Phi$, so that
 $$
 \phi_{\U}(S)=[\la S(\ib)\ra_{\ib\in \II}]_{\U}. $$
 Then $\phi_{\U}$
is a ring homomorphism whose kernel $\Ik=\ker\phi_{\U}$ is a prime 
ideal of $\RR$ that includes $\Ik_{0}$ and is disjoint from $\RP$.
Moreover the 
counting set $\II$
turns out to be \emph{suitable for} the ideal 
$\Ik$.

The map 
   $X\mapsto \phi_{\U}(S_{X})$ induces an equivalence
   relation between point sets
   $$X\equ Y\ \ \Iff\ \ \{\ib\in\II \mid S_{X}(\ib)= 
      S_{Y}(\ib)\}\in\U$$
that satisfies all conditions of  an equinumerosity relation but 
possibly Zermelo's Principle \ZP.

If the kernel $\Ik=\ker\phi_{\U}$ is a gauge 
ideal,  then 
$\equ$ is an equinumerosity relation 
 whose set of numerosities is 
(isomorphic to) a discrete semiring of hyperreal numbers, namely
\emph{a subsemiring  of the non-negative part of the ultrapower}
$\R\ult{\II}{\U}$. A simple 
rephrasing of the definition gives

\begin{proposition}\label{real}
The equivalence $\equ$ is an equinumerosity, or equivalently   
$\Ik=\ker\phi_{\U}$ is
a gauge ideal  of $\RR$, if and only if, for all $S\in\RR$,
$$\{ \ib\in\II\,\mid\, S(\ib)>0\,\}\in\U\ \ \Iff\ \ \exists P\in\RP 
\mbox{s.t.}\ \{ \ib\in\II\,\mid\, S(\ib)=P(\ib)\,\}\in\U.$$
\qed
  
\end{proposition}

Conversely we can state
\begin{theorem}\label{null}
Let $\eq$ be an equinumerosity, and let $\Ik$ be the corresponding 
gauge ideal of $\RR$. Assume that there exists a counting set $\II$ 
suitable for $\Ik$. 
Then  there exists an
ultrafilter $\U$ suitable for
$\II$ such that the equinumerosity $\eq$ coincides with the 
equivalence $\equ$ induced by the counting map $\phi_{\U}$.
   
So the set  of numerosities $\Nk$ of $\eq$ is isomorphic 
to a discrete subsemiring of the non-negative part of the ultrapower 
${\R}\ult{\II}{\U}$. 
  
\end{theorem}

\proof
First of all remark that the family $\F$ of the zero-sets 
$Z(S)=\{\ib\in\II\mid S(\ib)=0\}$ for 
$S\in\Ik$ has the finite intersection property, because $Z(S)\cap 
Z(T)=Z(S^{2}+T^{2})$. Let $\U$ be any ultrafilter containing $\F$.
We claim that  if $Z(S)\in\U$ then $S\in\Ik$.

If $S\notin\Ik$ we may assume \wlg\ that there exists $P\in\RP$ and 
$I\in\Ik$ such that $S=I+P$. Then  $Z(S)\cap Z(I)\incl Z(P)$, and 
$Z(P)\cap \II_{k}=\0$ for every sufficiently large $k$. But $\II_{k}$ 
is the intersection of the zero-sets of the polynomials $t_{n}-1$ for 
$n\le k$, so it belongs to $\U$, contradiction.
Hence $\Ik=\ker\phi_{\U}$, and the thesis follows.
\qed

\bigskip
In the case of natural equinumerosities, one can consider
only $0$-$1$ assignements, because only these annihilate the series
$$\sum_{n\in\N}(t_{n}^{2}-t_{n})^{2}.$$ Namely,
arrange the set of all eventually zero sequences of zeroes and ones
     in a sequence   
    $\LL=\la \xb_{F}\mid{F\in\N^{<\og}}\ra$, where $\xb_{F}(n)=1$ if and 
    only if $n\in F$. Then
    
\begin{lemma}\label{iso}
   The counting map $\Phi: \RR\to \Z^{\LL}$ such that $\Phi(S)=\la 
   S(\xb_{F})\mid F\in \N^{<\og}\ra$ is a surjective homomorphism of 
   partially ordered rings, whose kernel is the ideal $\Ik_{1}$ that
   consists of all those series 
   whose corresponding squarefree 
   series is $0$.
  
\end{lemma}
    
\proof The map $\Phi$ is a homomorphism by definition, and clearly it 
maps $\RP$ into $\N^{\LL}$.

Given $S=\sum n_{\aa}t^{\aa}\in \RR$ and $F\in\N^{<\og}$, recall that 
$supp(\aa)=\{n\in\N\mid a_{n}\ne 0\}\ $ and\ $\ t_{F}=\prod_{n\in 
          F}t_{n}$. Put  
$n_{F}=\sum_{F=supp(\aa)}n_{\aa}$: then 
$S(\xb_{F})=\sum_{E\incl F}n_{E}$, and 
$S'=\sum_{F\in\N^{<\og}}n_{F}t_{F}$. By the inclusion-exclusion 
principle one has
$n_{F}=\sum_{E\incl F}(-1)^{|F\7 E|}S(\xb_{E})$.
Hence $S\in \ker\Phi$ if and only if $n_{F}=0$ for all 
$F\in\N^{<\og}$, or equivalently $S'=0$.
In particular, no non-zero squarefree series lies in the kernel of 
$\Phi$.

On the other hand, given $g\in\Z^{\LL}$ one has that any series 
$S=\sum n_{\aa}t^{\aa}$ 
such that $n_{F}=\sum_{E\incl F}(-1)^{|F\7 E|}g({E})$ satisfies 
$\Phi(S)=g$. So $|n_{F}|\le \sum_{E\incl F}|g({E})|$. Put 
$d_{n}=\sum_{E\incl\{0,\ldots,n\}}|g({E})|$: then we can find in $\RR$
a series $S$ of degree not exceeding $d_{n}$ in each variable $t_{n}$ 
that satisfies the condition above.

\qed

We are now ready to classify all natural equinumerosity relations.
Call \emph{gauge} a fine\footnote{
~The ultrafilter $\U$ is fine if all cones $C_{n}=\{F\mid n\in F\}$ 
belong to $\U$.}
ultrafilter $\U$ over $\N^{<\og}$ if every 
square\footnote{
~Recall that every element of $\N^{\LL}$ is a sum of squares.}
$f^{2}\in\N^{\LL}$ is equivalent modulo $\U$ to a 
function $g$ such that $n_{F}=\sum_{E\incl F}(-1)^{|F\7 E|}g({E})\ge 0$ for 
all $F\in\N^{<\og}$. (Remark that $\Phi(\sum n_{F}t_{F}) =g.)$

\begin{theorem}\label{rat}
    Let $\eq$ be a natural equinumerosity, and let $\Ik$ be
    the corresponding gauge ideal  of $\RR$.
Then $\LL$ is suitable for $\Ik$, and there exists a unique gauge
 ultrafilter $\U$ over $\N^{<\og}$ suitable for $\LL$ such that 
\begin{equation}
    X\eq Y\ \ \Iff\ \ \{F\in\N^{<\og} \mid S_{X}(\xb_{F})= 
           S_{Y}(\xb_{F})\}\in\U.\tag{***}
    \label{***}
\end{equation}
    Hence the counting map $\Phi$ induces an ordered semiring 
    isomorphism between the set  of numerosities $\Nk$ of $\eq$ and
     the ultrapower 
    $\N\ult{\N^{<\og}}{\U}$. 
    
Conversely, the condition \emph{(\ref{***})} defines a natural equinumerosity on 
$\WW$ if and only if the ultrafilter $\U$ is gauge.   
  
\end{theorem}

\proof
The ideal $\Ik$ includes the ideal $\Ik_{1}$, hence in particular all 
series whose corresponding squarefree is $0$, \ie\  the kernel of $\Phi$. 
Let $\Jk=\Phi[\Ik]$ be the ideal of $\Z^{\LL}$ corresponding to 
$\Ik$. Then $\RR/\Ik$ is isomorphic to $\Z^{\LL}/\Jk$, and every 
element of $\Jk$ has some  component equal to $0$.
 In fact, assume that $f\in\Jk$ has no zeroes, and
  multiply $f$ by a suitable sequence of $\pm 1$, so as to obtain a 
  sequence $g\in\Jk$ of positive integers. Then $g$ can be written as 
  $$g=1+h_{1}^{2}+h_{2}^{2}+h_{3}^{2}+h_{4}^{2}=
  \Phi(1+S_{1}^{2}+S_{2}^{2}+S_{3}^{2}+S_{4}^{2}).$$
  But the series $1+S_{1}^{2}+S_{2}^{2}+S_{3}^{2}+S_{4}^{2}$ is equivalent 
  modulo $\Ik$ to an element of $\RP$, and so its image $g$ under 
  $\Phi$ cannot belong to $\Jk$.
  So the counting set 
$\LL$ is suitable for $\Ik$, and $\Jk$ is 
 a prime zero-ideal that is determined by an ultrafilter $\U$
over $\LL$, which in turn is suitable for $\LL$, because $\II_{k}$ is 
the intersection of the zero-sets of the polynomials $t_{n}-1$ for 
$0\le n\le k$. It follows that  $\U$ is a fine ultrafilter, and that 
the condition (\ref{***}) holds.

The ultrafilter $\U$ is gauge, because, given $f^{2}=\Phi(S^{2})>0$, 
there exists $P\in\RP$ that is equivalent to $S^{2}$ modulo $\Ik$. 
Put $\Phi(P)=g$: then clearly $g\equiv f^{2}$ modulo $\U$, and 
satisfies the equalities $\sum_{E\incl F}(-1)^{|F\7 E|}g({E})\ge 0$ 
because $P\in\RP$.

Conversely, if $\U$ is a gauge ultrafilter, then 
(\ref{***}) defines an equinumerosity relation by 
Proposition \ref{real}, which is natural because the corresponding 
ideal of $\RR$ includes $\Ik_{1}$.

\qed

In particular we have
\begin{cor}\label{tnat}
The condition $($\emph{\ref{***}}$)$ provides
a biunique correspondence between natural equinumerosity 
relations $\eq$ on $\WW$ and gauge ultrafilters $\U$ on $\N^{<\og}$, 
in such a way that the following diagram commutes

\bigskip
\begin{center}
\begin{picture}(185,70)
   \put(0,0){\makebox(0,0){$\Nk$}}
   \put(90,0){\makebox(0,0){$\RR/\Ik$}}
   \put(180,0){\makebox(0,0){$\Z^{\LL}/\Phi[\Ik]$}}
   \put(215,0){\makebox(0,0){$\cong\Z\ult{\N^{<\og}}{\U}$}}
   \put(0,70){\makebox(0,0){$\WW$}}
   \put(90,70){\makebox(0,0){$\RR$}}
   \put(180,70){\makebox(0,0){$\Z^{\LL}$}}
   \put(45,6){\makebox(0,0){$j$}}
   \put(132,6){\makebox(0,0){$\cong$}}
   \put(45,76){\makebox(0,0){$\Sigma$}}
   \put(135,76){\makebox(0,0){$\Phi$}}
   \put(-8,35){\makebox(0,0){$\nk$}}
   \put(15,0){\vector(1,0){60}}
    \put(105,0){\vector(1,0){54}}
   \put(15,70){\vector(1,0){60}}
   \put(105,70){\vector(1,0){60}}
   \put(0,60){\vector(0,-1){50}}
   \put(90,60){\vector(0,-1){50}}
   \put(180,60){\vector(0,-1){50}}
\end{picture}
\end{center}
\bigskip
In particular all sets of natural numerosities can be taken to be 
sets of hypernatural numbers of the corresponding ultrapowers 
$\N\ult{\N^{<\og}}{\U}$.
\qed
\end{cor}

\smallskip
\begin{remark}
    Put $F_{k}=\{0,\ldots,k\}$. In this context, the 
    \emph{asymptotic numerosities} defined 
    in \cite{QSU} can be characterized as the restrictions to $\WW_{0}$ 
    of those natural equinumerosities for which 
    $\II=\{\xb_{F_{k}}\mid k\in\N\}$ is a suitable counting set.
\end{remark}

\section{Final remarks and open questions}\label{froq}

It is interesting to remark that the natural numerosities satisfy a 
``Finite Approximation Principle'' analogous to that considered in 
\cite{bdf}. 
Let $\nk$ be a natural numerosity function, and for $X\in\WW$ and $F\in\N^{<\og}$ put 
$X_{F}=X\cap\bigcup_{n\in\N}F^{n}$, so that $S_{X}(\xb_{F})=|X_{F}|$.
Then

\begin{itemize}     
         \item[(\FAP)] ${~~~~}$ \emph{If $|X_{F}|\le |Y_{F}|$ for all $F\in\N^{<\og}$, 
	 then $\nk(X)\le\nk( Y)$.}
    \end{itemize}
    
    In fact, if $\U$ is the gauge ultrafilter corresponding to the 
    numerosity $\nk$, we have
    $$ 
	 \nk(X)\le\nk(Y)\ \ \Iff \ \ \{F\in\N^{<\og}\mid |X_{F}|\le 
	 |Y_{F}|\}\in\U. $$ 
	 
This fact suggests a ``Cantorian'' characterization of natural 
equinumerosity by means of a class of particular bijections, similar 
to the one obtained in \cite{QSU} for asymptotic equinumerosity. Call 
$\U$-\emph{congruence between $X$ and $Y$} a $1$-to-$1$ map 
$\tau:X\to Y$ such that
$\{F\in\N^{<\og}\mid \tau[X_{F}]= Y_{F}\}\in\U.$
Then we have
\begin{remark}\label{rat}
    Let $\eq$ be a natural equinumerosity, and let $\U$ be the 
    corresponding gauge
 ultrafilter  over $\N^{<\og}$. Then, by definition,
 \begin{center}
     \emph{if there exists a $\U$-congruence between 
         $X$ and $Y$, then $ X\eq Y$.}
 \end{center}  
The reverse implication seems difficult to prove in general. A 
simple inductive definition of $\tau$ can be given whenever the 
ultrafilter $\U$ contains a \emph{chain}. But in this case the 
equinumerosity becomes asymptotic after an appropriate reordering of 
$\N$.    
\end{remark}

When the ultrafilter is Ramsey the situation is much simpler, namely

\begin{cor}\label{selre}
A fine Ramsey ultrafilter $\U$ on $\N^{<\og}$ provides a natural equinumerosity 
relation on $\WW$.
The corresponding set of numerosities $\N\ult{\N^{<\og}}{\U}$ is isomorphic to the
ultrapower $\N\ult{\N}{\sg\U}$, where $\sg:\N^{<\og}\to\N$ is defined 
by $\sg(F)=|F|$.
This equinumerosity is characterized by the class of the 
$\U$-congruences, and becomes
asymptotic after a suitable reordering of $\N$.
\end{cor}

\proof
Consider the coloring of $[\N^{<\og}]^{2}$ given by $c(\{F,G\})=0$ if $F$ 
is comparable with $G$, $c(\{F,G\})=1$ otherwise. The ultrafilter 
$\U$ being Ramsey, it contains a homogeneous subset $H$, which is a 
chain, because $\U$ is fine.

In order to prove the first assertion, we show that $\U$ is gauge. 
Remark that, $\U$ being Ramsey,
any non-negative 
function $f:\N^{<\og}\to\Z$ is nondecreasing on some  subchain 
$H_{0}\pincl H_{1}\pincl\ldots\pincl H_{n}\pincl\ldots$ of $H$ that 
belongs to $\U$.  
Define the series 
$$S=f(H_{0})t_{H_{0}}+ 
\sum_{n\in\N}(f(H_{n+1})-f(H_{n}))t_{H_{n+1}}\in\RP:$$
then $\Phi(S)$ is equivalent to $f$ modulo $\U$, and so $\U$ is gauge.

Assume \wlg\ that $H$ is complete, \ie\ that the map 
$\sg$ restricted to $H$ is onto $\N$. Then $\sg$ induces an isomorphism between the 
ultrapowers $\Z\ult{\N^{<\og}}{\U}$ and $\Z\ult{\N}{\sg\U}$, and the 
second assertion is proved.

Given equinumerous sets $X$ and $Y$, let 
$H_{0}\pincl H_{1}\pincl\ldots\pincl H_{n}\pincl\ldots$ be a subchain of $H$
such that $|X_{H_{n}}|=|Y_{H_{n}}|$ for all $n\in\N$. Define 
$\tau:X\to Y$ by glueing together disjoint bijections 
$\tau_{0}:X_{H_{0}}\to Y_{H_{0}}$ and
$\tau_{n+1}:X_{H_{n+1}}\7 X_{H_{n}}\to Y_{H_{n+1}}\7 Y_{H_{n}}$ for 
all $n\in\N$.  Then $\tau$ is a $\U$-congruence between $X$ and $Y$.

Finally, any complete chain $H$ in $\U$ provides a reordering of $\N$ 
such that $H_{k}$ becomes $F_{k}$: hence,
with respect to this reordering, the equinumerosity is asymptotic.

\qed

\smallskip
Remark that  the ultrafilter $\sg\U$ on $\N$ defined above is Ramsey,
whenever the gauge ultrafilter $\U$ contains a (complete) 
chain,
because every non-negative function is non decreasing modulo $\sg\U$.
So natural equinumerosities exist under mild set theoretic 
hypotheses, namely those that provide Ramsey ultrafilters over $\N$, 
and the corresponding numerosities are, up to isomorphism, the 
\emph{asymptotic numerosities} of \cite{QSU}.
The question as to whether there exist non-Ramsey gauge ultrafilters 
is still open, and with it the most interesting question of the 
existence of natural equinumerosities in \zfc\ alone. However we 
conjecture that \emph{only P-point ultrafilters can be gauge}, and so 
the question would be solved in the negative. Combined with the 
``geometric''  intuition that the size of the diagonal might be 
different from that of the side, this conjecture is the reason why we 
did not include the Natural Transformation Principle \NP\ in the 
definition of equinumerosity.

 On the other hand, the existence of gauge ideals of the ring $\RR$ 
 seems to  be weaker 
than that of gauge ultrafilters, and so the existence of 
(non-natural) equinumerosities might be provable in \zfc.
However also this question is  
open up to now.


\bigskip
\bigskip

\bibliographystyle{amsplain}


\end{document}